\documentclass[12pt]{article}
\usepackage{latexsym}
\usepackage{amsmath,amsthm,amssymb,enumerate}
\def\hH{\wh{H}}

\def\deg{\text{deg}}

\newcommand{\hatt}[1]{\widehat #1}

\def\hR{\hatt{R}}

\def\a{\alpha} 
\def\b{\beta} 

\def\e{\epsilon} \def\f{\phi} \def\F{{\Phi}}  \def\g{\gamma}
\def\G{\Gamma}  \def\k{\kappa}
\def\z{\zeta}     \def\l{\lambda}
   \def\p{\pi}
\def\r{\rho}  \def\s{\sigma} 
\def\t{\tau} \def\om{\omega}  \def\Om{\Omega}

\def\cT{{\cal T}}

\newtheorem{theorem}{Theorem}
\newtheorem{lemma}[theorem]{Lemma}

\newtheorem{Remark}{Remark}

\def\cW{{\mathcal W}}

\def\cT{{\mathcal T}}

\newcommand{\ul}[1]{\mbox{\boldmath$#1$}}
\newcommand{\ol}[1]{\overline{#1}}
\newcommand{\wh}[1]{\widehat{#1}}

\newcommand{\rdown}[1]{{\left\lfloor #1\right \rfloor}}

\newcommand{\brac}[1]{\left(#1\right)}

\newcommand{\bfrac}[2]{\left(\frac{#1}{#2}\right)}

\newcommand{\ra}{\rightarrow}
\newcommand{\rat}{{\textstyle \ra}}

\newcommand{\set}[1]{\left\{#1\right\}}

\def\E{\mbox{{\bf E}}}

\def\Pr{\mbox{{\bf Pr}}}

\newcommand{\ignore}[1]{}

\newcommand{\cA}{{\cal A}}

\newcommand{\card}[1]{\left|#1\right|}

\newcommand{\beq}[2]{\begin{equation}\label{#1}#2\end{equation}}

\def\cG{\mathcal{G}}

\def\bd{{\bf d}}

\def\num{s}
\def\w{{\text w}}
\def\cZ{{\cal Z}}

\newcommand{\multstar}[1]{\begin{multline*}#1\end{multline*}}
\newcommand{\mult}[2]{\begin{multline}\label{#1}#2\end{multline}}

\author{Colin Cooper\thanks{Research supported in part by EPSRC grant EP/M005038/1} \and Alan Frieze\thanks{Research supported in part by NSF grant DMS1362785
} \and Wesley Pegden\thanks{Research supported in part by NSF grant DMS1363136}}

\begin{document}
\title{On the cover time of dense graphs}
\maketitle
\begin{abstract}
  We consider arbitrary graphs $G$ with $n$ vertices and minimum degree at least $\delta n$ where $\delta>0$ is constant.\\
(a) If the conductance of $G$ is sufficiently large then we obtain an asymptotic expression for the cover time $C_G$ of $G$ as the solution to an explicit transcendental equation.\\
(b)  If the conductance is not large enough to apply (a), but the mixing time of a random walk on $G$ is of a lesser magnitude than the cover time, then we can obtain an asymptotic deterministic estimate via a decomposition into a bounded number of dense subgraphs with high conductance. \\
(c) If  $G$ fits neither (a) nor (b) then we give a deterministic asymptotic (2+o(1))-approximation of $C_G$.
\end{abstract}

\section{Introduction}
Let $G=(V,E)$ be a connected graph with vertex set  $V=[n]=\set{1,2,\ldots,n}$ and an edge set $E$ of $m$ edges. In a simple random walk $W$ on a graph $G$, at each step, a particle moves from its current vertex to a randomly chosen neighbour. For $v\in V$, let $C_v$ be the expected time taken for a simple random walk  starting at $v$ to visit every vertex of $G$. The {\em vertex cover time} $C_G$ of $G$ is defined as $C_G=\max_{v\in V}C_v$. The (vertex) cover time of connected graphs has been extensively studied. It is a classic result of Aleliunas, Karp, Lipton, Lov\'asz and Rackoff \cite{AKLLR} that $C_G \le 2m(n-1)$. It was shown by  Feige \cite{Feige1}, \cite{Feige2}, that for any connected graph $G$, the cover time satisfies $(1-o(1))n\log n\leq C_G\leq (1+o(1))\frac{4}{27}n^3.$ As an example of a graph achieving the lower bound, the complete graph $K_n$ has cover time determined by the Coupon Collector problem. The {\em lollipop} graph consisting of a path of length $n/3$ joined to a clique of size $2n/3$ gives the asymptotic upper bound for the cover time.

It follows from \cite{AKLLR} that there is a very simple randomised algorithm for estimating the cover time. Simply execute enough random walks so that the average cover time can be used as an estimate. It is much more challenging to estimate the cover time deterministically in polynomial time. A theorem of Matthews \cite{Matt} gives a deterministic $O(\log n)$ approximation. This was improved to $O((\log\log n)^2)$ by Kahn, Kim, Lov\'asz and Vu \cite{KKLV}. In a breakthrough, Ding, Lee and Peres \cite{DLP} improved this to $O(1)$ using a remarkable connection between the cover time and Gaussian free fields. Subsequently, Ding \cite{Ding} has improved the factor of approximation to $1+o(1)$, as $n\to\infty$ for trees and bounded degree graphs. Zhai \cite{Z} has recently shown that if the maximum hitting time is asymptotically smaller than the cover time then the approximation ratio is $1+o(1)$, implying the results of \cite{DLP} and \cite{Ding}.  An important point to note here is that Meka \cite{Meka} gives a polynomial time approximation scheme for finding the supremum of a Gaussian process. This is what provides the computational underpinning for the results of \cite{DLP}, \cite{Ding} and \cite{Z}. We note that none of these results give an explicit value of the  cover time as a function of the number of vertices $n$ or imply a deterministic polynomial time approximation scheme for the cover time.

The first two authors of this paper have studied the cover time of various models of a random graph, see \cite{CFreg,CFweb,CFgiant,Dnp}. The main tool in their analysis has been the ``First Visit Lemma'', see Lemma \ref{MainLemma}. In this paper we see how this lemma can be used  deterministically to give good estimates of the cover time of dense graphs when the mixing time is asymptotically smaller than the cover time.

Let $\cG(n,\theta)$ denote the set of connected graphs with vertex set $[n]$ and minimum degree at least $\theta n$. Our first result deals with the simplest case, where the mixing time of a random walk on our graph is sufficiently small. Subsequent theorems will consider more general cases.

{\bf Notation:} The degree sequence of the graph $G=(V,E),\,|V|=n,$ will always be $\bd=(d_1,d_2,\ldots,d_n)$ so that $2m=\sum_{i=1}^nd_i$. For $S\subseteq V$ we let $d(S)=\sum_{i\in S}d_i$ and $e(S)=\set{\set{v,w}\in \binom{S}{2}\cap E}$.

For two sequences $A_n,B_n$ we write $A_n=(1\pm\e)B_n$ if $(1-\e)B_n\leq A_n\leq (1+\e)B_n$ for $n$ sufficiently large. For two sequences $A_n,B_n$ we write $A_n\approx B_n$ if $A_n=(1+o(1))B_n$ as $n\to\infty$. We will write $A_n\gg B_n$ or $B_n\ll A_n$ to mean that $A_n/B_n\to\infty$ as $n\to\infty$.

For $S,T\subseteq V,\,S\cap T=\emptyset$  let $e(S,T)=e_G(S,T)=\set{\set{v,w}\in E:v\in S,w\in T}$,  let $\ol S=V\setminus S$ and
\begin{equation}\label{CondS}
\F(S)=\F_G(S)=\frac{e(S,\ol S)d(V)}{d(S)d(\ol S)}.
\end{equation}

The {\em conductance } $\F(G)$ of $G$ is given by
$$\F(G)=\min_{0<d(S)\leq m}\F_G(S).$$
We will make our walk {\em lazy} and ergodic by adding a loop at each vertex so that the walk stays put with probability 1/2 at each step. This has the effect of (asymptotically) doubling the cover time and the extra factor of two can be discarded. (Ergodicity only requires a small probability of staying in place, but laziness allows us to use conductance to estimate the mixing time. See \eqref{mix}.) A simple random walk has a steady state of
\beq{steady}{
\p_i=\frac{d_i}{2m},i\in [n] \text{ and if } G\in\cG(n,\theta) \text{ then }\frac{\theta}{n}\leq \p_i\leq \frac{1}{\theta n}.
}
Next let
\beq{deft}{
F(t)=\sum_{v\in V}\frac{e^{-\p_vt}}{\p_v}\quad\text{ and so }\quad F'(t)=-\sum_{v\in V}e^{-\p_vt}.
}
Note that $F$ is monotone decreasing and $F'$ is monotone increasing and that $F$ is convex.
Next let
\beq{e0}{
\psi=\frac{1}{\log^{2/3}n}.
}
\begin{theorem}\label{th1}
Let $\e>0$ be arbitrary and suppose that $G\in \cG(n,\theta)$ where $\theta=\Omega(1)$. Suppose that $\F=\F(G)\geq n^{-\theta\psi}$. Then there exists $n_\e$ such that if $n\geq n_\e$ then
\begin{equation}\label{ft1}
C_G=(1\pm\e)t^*
\end{equation}
where $t^*$  is the unique solution to $F'(t)=-1$,  (see \eqref{deft}).
\end{theorem}
Thus, if $G$ is regular and satisfies the conditions of Theorem \ref{th1} then $C_G\approx n\log n$. Also, if $G$ is regular of degree $\theta n$ where $\theta>1/2$, then the conditions of Theorem \ref{th1} will be satisfied. Indeed,  the condition that $d(S)\leq m$ in the definition of conductance is equivalent to $|S|\leq n/2$ and then
$$\F(S)\geq \frac{\brac{\theta-\frac12}n|S|d(V)}{n|S|d(V)}=\theta-\frac12.$$
\begin{Remark}\label{reem}
Note also, that while it may be difficult to compute $\F(G)$ exactly in deterministic polynomial time, we can approximate it to within an $O(\log n)$ factor using the algorithm of Leighton and Rao \cite{LR}. Thus if $\F(G)\gg n^{-\theta\psi}\log n$ then there is a deterministic polynomial time algorithm that verifies that $G$ satisfies the conditions of Theorem \ref{th1} and  gives  a $(1+\e)$-approximation to the cover time.
\end{Remark}
The proof of Theorem \ref{th1} is given in Section \ref{easy} and closely follows the lines of the proofs for random instances.

Suppose that we start our walk $\cW_u=(u=\cW_u(0),\cW_u(1),\ldots,\cW_u(t),\ldots)$ at vertex $u$ and that $P_{u}^{(t)}(x)=\Pr(\cW_u(t)=x)$. Let 
\[
d(t)=\max_{u,x\in V}|P_{u}^{(t)}(x)-\pi_x| ,
\] 
and let $T_{mix}=T_{mix}(\om)$ be such that, for $t\geq T_{mix}$
\begin{equation}\label{mix0}
\max_{u,x\in V}\card{\frac{P_{u}^{(t)}(x)-\pi_x}{\pi_x}} \leq \frac{1}{\om}
\end{equation}
where $\om=\om(n)\to\infty$. We will assume that
\beq{defom}{
\om=n^{3\theta\psi}.
}
If the conditions of Theorem \ref{th1} fail, then we  partition the vertex set $V$ into $O(1)$ subsets which satisfy the conditions of Theorem \ref{th1}. If furthermore,  our {\em mixing time}
$$T_{mix}=o(C_G)$$
then we will obtain a $(1+\e)$-approximation to the cover time.

\begin{Remark}\label{reem1}
Note that by examining the powers of the transition matrix $P$, we can determine the mixing time $T_{mix}(u,\om),u\in V$ in deterministic polynomial time. We note that $T_{mix}(u)=O(n^3)$, as long as the accuracy  needed in \eqref{mix0} is at most $ 1/\om=e^{-poly(n)}$, see \cite{LPW} (Proposition 10.28).  In which case we  only need to compute a $poly(n)$ power of $P$.
\end{Remark}

\begin{theorem}\label{th2}
Let $\e>0$ be arbitrary and suppose that $G\in \cG(n,\theta)$ where $\theta=\Omega(1)$. Then in deterministic polynomial time we can find a partition of $V$ into subsets $V_i,i=1,2,\ldots,s$ where $s=O(1)$, where the induced subgraphs $G[V_i]$ satisfy the conditions of Theorem \ref{th1}, and have cover time $C_i$ which can be computed via Theorem \ref{th1}.

Suppose {furthermore that} $T_{mix}=o(C)$, where $C$ is given by
\beq{CG}{
C=\max\set{\frac{C_i}{\p(V_i)}:i\in [s]},
}
then $C_G=(1\pm\e) C$ .
\end{theorem}
The construction of this partition is described in Section \ref{Pgraph}.

Finally, if $T_{mix}$ is too large for Theorem \ref{th2} to apply then we do not have a nice expression for $C_G$, but instead we have

\begin{theorem}\label{th3}
Let $\e>0$ be arbitrary and suppose that $G\in \cG(n,\theta)$ where $\theta=\Omega(1)$. Then in deterministic polynomial time we can compute an estimate $\ol{C}_G$ such that if $n\geq n_\e$ then
\beq{eth3}{
\ol{C}_G\leq C_G\leq (2+o(1))\ol{C}_G.
}
\end{theorem}
The proof of Theorem \ref{th2} uses a concentration inequality of Paulin \cite{Paulin}, which requires a sufficiently small mixing time. The proof of Theorem \ref{th3} uses the partition of Theorem \ref{th2}. It then describes how to use the transition matrix of the walk to give upper and lower estimates for the time needed to visit each $V_i$.
\section{First Visit Lemma}
Our main tool will be Lemma \ref{MainLemma} below. The lemma has been used several times in the context of random graphs, see for example \cite{CFreg,CFweb,CFgiant,Dnp}. We sharpen the proof to make it fit the current situation. Let $G$ denote a fixed connected graph, and let $u$ be some arbitrary vertex from which a walk $\cW_{u}$  is started. Let $\cW_{u}(t)$ be the vertex reached at step $t$, let $P$ be the matrix of transition probabilities of the walk, and let
$$h_t=P_{u}^{(t)}(v)=\Pr(\cW_{u}(t)=v).$$

It follows from e.g. Aldous and Fill \cite{AF}, Lemma 2.20, that $d(t)$ satisfies $d(s+t)\leq 2d(s)d(t)$  which implies that
\[
\max_{u,x\in V}|P_u^{(kT)}(x)-\p_x|\leq 2^{k-1}(\max_{u,x\in V}|P_u^{(T)}(x)-\p_x|)^k\leq 2^{k-1}\bfrac{\p_x}{\om}^k.
\]
And because $d(t)$ is monotone decreasing in $t$, for $t\geq T=T_{mix}$ and $k=\rdown{t/T}$, we have
\begin{equation}\label{4a}
\max_{u,x\in V}\card{\frac{P_{u}^{(t)}(x)-\pi_x}{\pi_x}} \leq \frac{2^{k-1}}{\om^{k}}.
\end{equation}
Fix two vertices $u,v$. Let
\begin{equation}\label{Hz}
H(z)=\sum_{t=T}^\infty h_tz^t
\end{equation}
generate $h_t$ for $t\geq T$.

Next, considering the  walk $\cW_v$, starting at $v$, let $r_t=\Pr(\cW_v(t)=v)$ be the probability  that this  walk returns to $v$ at step $t = 0,1,...$. Let
$$R(z)=\sum_{t=0}^\infty r_tz^t$$
generate $r_t$. Our definition of return includes $r_0=1$.

For $t\geq T$ let $f_t=f_t(u \rat v)$ be  the probability that the first visit of the walk $\cW_u$ to $v$ in the period $[T,T+1,\ldots]$ occurs at step $t$. Let
$$F(z)=\sum_{t=T}^\infty f_tz^t$$
generate  $f_t$.
Then we have
\begin{equation}
\label{gfw} H(z)=F(z)R(z).
\end{equation}
Finally, let
\begin{equation}
\label{Qs}
R_T(z)=\sum_{j=0}^{T-1} r_jz^j\text{ and }H_T(z)=\sum_{j=0}^{T-1} h_jz^j.
\end{equation}
Now fix $u\neq v\in V$. For a large constant $K>0$, let
\begin{equation}
\label{lamby}
\l=\frac{1}{KT}.
\end{equation}
For $t\geq 0$, let $\cA_t(v)$ be the event that $\cW_u$ does not visit $v$ in steps $T,T+1,\ldots,t$. The vertex $u$ will have to be implicit in this definition.
\begin{lemma}\label{MainLemma}
Suppose that
\begin{description}
\item[(a)]
For some constant $c >0$, we have
\beq{RTz}{
\min_{|z|\leq 1+\l}|R_T(z)|\geq c.
}
\item[(b)]
\beq{Tpi}{
T\pi_v\leq \om^{-1}=o(1).
}
\end{description}
Let $R_v=R_T(1)$. Then we can write
\begin{equation*}
p_v=\frac{\pi_v}{R_v(1+\xi_{v,1})}\text{ where }|\xi_{v,1}|=O(\om^{-1}).
\end{equation*}
And then for all $t\geq T$,
\begin{equation}
\label{frat}
\Pr(\cA_t(v))=\frac{1+\xi_{v,2}}{(1+p_v)^{t}} +o(Te^{-\l t/2})\text{ where }|\xi_{v,2}|=O(\om^{-1}).
\end{equation}
\end{lemma}
\begin{proof}
The proof is very similar to that given in previous papers. We will defer its proof to an appendix.
\end{proof}
\begin{Remark}\label{rem1a}
We will not have to verify \eqref{RTz} to use the theorem. It was shown in \cite{Hyper} that \eqref{RTz} follows from $R_v=O(1)$ and in our applications, $R_v=1+o(1)$.
\end{Remark}

\section{Proof of Theorem \ref{th1}}\label{easy}
Because our results require $n\geq n_\e$ we can state inequalities in asymptotic terms. I.e. if we want to show that some parameters $A_n,B_n$ satisfy
$A_n\leq (1+\e)B_n$ then we can write $A_n\leq (1+o(1))B_n$. Then if $n$ is large enough, so that the $o(1)$ term is at most $\e$, then we are dealing with a bounded size problem, which can in principle, be dealt with by an exponential time algorithm.

We continue by computing parameters for  use in Lemma \ref{MainLemma}. We begin with the mixing time $T=T_{mix}$. We use the following Cheeger inequality, see for example Levin, Peres and Wilmer \cite{LPW}, (13.6).
\beq{mix}{
\max_{u,x\in V}\card{\frac{P_{u}^{(t)}(x)-\pi_x}{\pi_x}}\leq \frac{e^{-\F^2t/8}}{\min_u\p_u}\leq \theta^{-1}ne^{-\F^2t/8},
}
where the last inequality follows from \eqref{steady}. (We have $e^{-\F^2t/8}$ instead of $e^{-\F^2t/2}$ because our definition of $\F$ is larger than that defined in (7.8) of \cite{LPW}, but larger by a factor of at most 2.)

We can satisfy \eqref{mix0} if we take
\beq{defT}{
T=\frac{8\log(\om n/\theta)}{\F^2}=\frac{8\log(n^{1+3\theta\psi}/\theta)}{\F^2}\leq \frac{10\theta\log n}{n^{-2\theta\psi}}\leq n^{3\theta \psi}.
}
With this value of $T$ we find that
\beq{withT}{
T\p_v\leq \frac{T}{\theta n}\leq \frac{n^{3\theta\psi}}{\theta n}\leq \om^{-1}=n^{-3\theta\psi}.
}
We therefore find that \eqref{Tpi} is satisfied.
\begin{lemma}\label{tstar}
Let $t^*$ be as in Theorem \ref{th1}. Then,
\begin{enumerate}[(a)]
\item
\beq{1}{
F(t^*)=o(t^*).
}
\item
\beq{bounds}{
n\log n\leq t^*\leq \theta^{-1}n\log n.
}
\end{enumerate}
\end{lemma}
\begin{proof}
Now, by convexity,
\[
\frac{1}{n}=\sum_{v\in V} \frac1n e^{-t^*\pi_v} \ge e^{-\frac{t^*}{n} \sum \pi_v} =e^{-t^*/n}.
\]
and this implies  the lower bound in \eqref{bounds}.

Next observe that \eqref{steady} implies
$$1=\sum_{v\in V} e^{-t^*\p_v}\leq ne^{-{ \theta t^*/n}}$$
and this implies the upper bound in \eqref{bounds}. Also we have that
$$\frac{F(t^*)}{t^*}\leq \sum_{v\in V}\frac{e^{-t^*\p_v}}{\theta\log n}=\frac{1}{\theta\log n}=o(1),$$
as claimed in \eqref{1}.
\end{proof}
\subsection{Upper bound on $C_G$}\label{upp}
We consider the walk $\cW_u$ and write $T=T_{mix}(u)$ for the mixing time. We observe first that
\beq{Rv}{
1\leq R_v\leq 1+\frac{T}{\theta n}=1+o(1).
}
The inequality follows from the fact that if the walk $\cW_v$ is not at $v$ then the probability it moves to $v$ at the next step is at most $1/\theta n$. The final claim can be seen from \eqref{withT}. Here we have ignored the self loops added to each vertex to make the chain lazy. As already mentioned, the addition of these loops multiplies the covertime by $(2+o(1))$ and so to get the covertime we would multiply and then divide by this factor. Here we just acknowledge that it multiplies the mixing time by a factor $(2+o(1))$, which can also be ignored in the above equation.

Let $T_{cov}(u)$ be the time taken to visit every vertex of $G$ by the random walk $\cW_u$. Let $U_t$ be the number of vertices of $G$ which have not been visited by $\cW_u$ at step $t$. We note the following:
\begin{align}
\Pr(T_{cov}(u) > t)&=\Pr(U_t\geq 1)\le \min\{1,\E(U_t)\}\label{TG},\\
C_u=\E(T_{cov}(u))&=\sum_{t > 0} \Pr(T_{cov}(u) \geq t) \label{ETG}
\end{align}
It follows from \eqref{TG}, \eqref{ETG} that for all $t$
\begin{equation}\label{shed}
C_u \le t+1+ \sum_{s> t} \E(U_s)\leq t+1+\sum_{v\in V}\sum_{s> t}
\Pr(\cA_s(v)).
\end{equation}
Putting $t=t^*$, defined in \eqref{ft1},  we see from \eqref{frat} that
\begin{align}
C_u& \le t^*+1+ \sum_{v\in V}\sum_{s> t^*}\brac{\frac{(1+\xi_{v,2})}{(1+p_v)^s}+o(Te^{-\l s/2})}\nonumber\\
   &=t^*+1+ \sum_{v\in V}\brac{\frac{(1+\xi_{v,2})}{p_v(1+p_v)^{t^*+1}}+o(T^2e^{-\l t^*/2})}\nonumber\\
  &=t^*+1+ \sum_{v\in V}\brac{\frac{(1+\xi_{v,2})\exp\set{-(t^*+1)\log(1+p_v)}}{p_v}+o(T^2e^{-\l t^*/2})}\nonumber\\
&=t^*+1+(1+O(\om^{-1}))\sum_{v\in V}\brac{\frac{e^{-p_vt^*+O(p_v^2t^*)}}{p_v}+o(T^2e^{-\l t^*/2})}.\label{upper}
\end{align}
\begin{Remark}\label{observe}
Observe that the term $o(T^2e^{-\l t^*/2})$ is negligible, since $t^*=\Theta(n\log n)$ and $\l=\Omega(n^{-3\e\theta})$. It is in fact at most $e^{-n^{1-3\e\theta}}$ and we will assume always that $\e$ is sufficiently small.
\end{Remark}
Now, because $t^*=\Theta(n\log n)$, we have, using \eqref{withT} and \eqref{steady},
$$p_vt^*=\brac{1+O\bfrac{T}{\theta n}+O\bfrac{1}{n}}\p_vt^*=\p_vt^*+O(n^{-1+3\theta\psi}\log n)\text{ and }p_v^2t^*=O\bfrac{\log n}{n}.$$
And so we can replace \eqref{upper} by
\beq{upperx}{
C_G\leq t^*+1+(1+O(\om^{-1}))F(t^*)+O(e^{-n^{1-3\e\theta}})\leq \brac{1+\frac{2}{\theta\log n}}t^*\leq (1+\e)t^*,
}
after using \eqref{1}.
\begin{Remark}\label{rem1}
Using Remark \ref{observe} we obtain a simpler upper bound:
$$\Pr(T_{cov}(u)\geq t)\leq \sum_{v\in V}\Pr(\cA_t(v))\leq (1+o(1))\sum_{v\in V}\frac{e^{-(1+o(1))t\p_v}}{\p_v}.$$

Putting $t=Kt^*$, we see that for any constant $L>0$ there exists $K=K(L)$ such that
$$\Pr(T_{cov}(u)\geq Kt^*)\leq n^{-L}.$$
\end{Remark}
\subsection{Lower bound}
Now let $T,u$ be as in Section \ref{upp}.
\beq{t1}{
t_1=t^*(1-\e_1)\text{ where }\e_1=\frac{1}{\log^{1/2}n}.
}
Then let $U_1$ denote the set of vertices that have not been visited by $\cW_u$ by time $t_1$, and let  $\cT=\set{\cW_u(i):1\leq i\leq T}$. Then we have that
\beq{-T}{
\E(|U_1|)=\sum_{v\in V}\Pr({(v\notin\cT)}\wedge \cA_{t_1}(v))\geq -T+\sum_{v\in V}\Pr(\cA_{t_1}(v)).
}
Here we subtract $T$ to account for visits before the mixing time $T$.

Applying Lemma \ref{MainLemma} we see that
\begin{align*}
\E(|U_1|)&\geq -T+\sum_{v\in V}\brac{\frac{1+\xi_{v,2}}{(1+p_v)^{t_1}}+o(Te^{-\l t_1/2})}\\
&=-T+(1-o(1))\sum_{v\in V}e^{-\p_v(1-\e_1)t^*}\\
&\geq -T+(1-o(1))e^{\e_1\theta t^*/n}\sum_{v\in V}e^{-\p_vt^*}\\
&\geq -T+(1-o(1))n^{\e_1\theta}\\
&\approx n^{\e_1\theta}\to\infty,
\end{align*}
after using \eqref{bounds} to  lower bound $e^{\e_1\theta t^*}$ and \eqref{defT} to bound $T$, (here $\psi=o(\e_1)$).

We summarise this as
\beq{sumA}{
\E(|U_1|)\approx \sum_{v\in V}\Pr(\cA_{t_1}(v))\approx n^{\e_1\theta}.
}
We now use the second moment method to show that $|U_1|>0$ w.h.p. Fix two vertices $v,w$ distinct from the start $u$ of the walk. Let $\G=\G_{v,w}$ be obtained from $G$ by identifying $v,w$ as a single vertex $\g=\g_{v,w}$ and keeping the loop if $\set{v,w}\in E(G)$.

There is a natural measure preserving map from the set of walks in $G$ which start at $u$ and do not visit $v$ or $w$, to the corresponding set of walks in $\G$ which do not visit $\g$. Thus the probability that $\cW_u$ does not visit $v$ or $w$ in the first $t$ steps is equal to the probability that a random walk $\wh{\cW}_u$ in $\G$ which also starts at $u$ does not visit $\g$ in the first $t$ steps.

We first check that Lemma \ref{MainLemma} can be applied to $\wh{\cW}_u$. We observe that it is valid to use $T$ as a mixing time. This follows from Corollary 3.27 of \cite{AF} viz. that the relaxation time of a collapsed chain is bounded from above by that of the uncollapsed chain. Our estimate for $R_\g$ should now be $1+O(2T/(\theta n))$ (the 2 coming from vertices that are neighbors of $v$ and $w$ in $G$). Now
$$
\frac{\p_\g}{R_\g}=\frac{\p_v+\p_w}{1+O(T/(\theta n))}=\brac{1+O\bfrac{T}{\theta n}}\brac{\frac{\p_v}{R_v}+\frac{\p_w}{R_w}}.$$
And so since $t_1=\Theta(n\log n)$ we have
\multstar{
  \Pr(v,w\in U_1)=\Pr((v,w\notin \cT)\wedge \cA_{t_1}(v)\wedge \cA_{t_1}(w))=\Pr((\g\notin\cT)\wedge \cA_{t_1}(\g))\leq \Pr(\cA_{t_1}(\g))\\
=(1+O(\om^{-1})) \exp\set{-\frac{\p_\g}{R_\g}t_1}=(1+O(\om^{-1})) \exp\set{-\brac{\frac{\p_v}{R_v}+\frac{p_w}{R_w}}t_1}\\
=(1+O(\om^{-1})) \Pr(\cA_{t_1}(v))\Pr(\cA_{t-1}(w)).
}
It follows therefore that after using \eqref{sumA},
$$\E(|U_1|^2)\leq \E(|U_1|)+(1+O(\om^{-1}))\E(|U_1|)^2.$$
So, by the Chebyshev inequality,
\beq{lbU}{
\Pr(|U_1|=0)\leq \frac{\E(|U_1|^2)-\E(|U_1|)^2}{\E(|U_1|)^2}\leq \frac{1}{\E(|U_1|)}+\frac{1}{\om}\leq \frac{2}{n^{\e_1\theta}}.
}
This implies that $C_G\geq  (1-o(1))t_1$ and completes the proof of Theorem \ref{th1}.

We will need the following lemma in Section \ref{Tlarge}. Let
$$\e_2=\frac{1}{\log^{1/4}n}.$$
\begin{lemma}\label{rem2}
$$\Pr\brac{\card{T_{cov}(u)-t^*}\geq \e_2t^*}\leq 3\e_2.$$
\end{lemma}
\begin{proof}
The probabilistic lower bound for $T_{cov}(u)$ follows from \eqref{lbU}. For the upper bound, for a given $\a>0$, we let $P_\a=\Pr(T_{cov}(u)\geq (1+\a)t^*)$ and then we have for some large constant $K>0$ that from \eqref{upperx},
\mult{PP}{
\brac{1+\frac{2}{\theta\log n}}t^*\geq \E(T_{cov}(u))\geq\\
\E(T_{cov}(u)\mid T_{cov}(u)\leq t_1)\Pr(T_{cov}(u)\leq t_1)+\\
\E(T_{cov}(u)\mid t_1<T_{cov}(u)\leq (1+\a)t^*)\Pr(t_1<T_{cov}(u)
\leq (1+\a)t^*)\\
+\E(T_{cov}(u)\mid (1+\a)t^*<T_{cov}(u)\leq Kt^*)\Pr((1+\a)t^*<T_{cov}(u)\leq Kt^*) \\
\geq  0+t_1\brac{1-\frac{2}{n^{\e_1\theta}}-P_\a}+(1+\a)t^*(P_\a-O(n^{-L})).
}
Here $K$ and $L$ are related as in Remark \ref{rem1}. We obtain
$$\Pr(t_1<T_{cov}(u)\leq (1+\a)t^*)\geq \brac{1-\frac{2}{n^{\e_1\theta}}-P_\a}$$
from \eqref{lbU}.

It follows from \eqref{PP}, after division by $t^*$, that
$$1+\frac{2}{\theta\log n}\geq 0+\brac{1-\frac{2}{n^{\e_1\theta}}-\frac{2}{\log^{1/2}n}-P_\a}+(1+\a)P_\a.$$
We deduce from this that $\a P_\a\leq \frac{3}{\log^{1/2}n}$ and then that $P_\a\leq \frac{3}{\log^{1/4}n}$ for $\a=\frac{1}{\log^{1/4}n}$.
\end{proof}

\section{Partitioning the graph}\label{Pgraph}

{\bf Notation:} For  sets $S\subseteq X\subseteq V$, let $\deg_X(v)$ denote the number of neighbors of $v$ in $X$, and let $\deg_X(S)=\sum_{v\in S}\deg_X(v)$. We will reserve the un-subscripted $\deg$ for $\deg_V$. For given $S\subseteq X\subseteq V$, we also use $X$ as the subgraph $G[X]$ of $G$ induced by $X$ in the notation $\F_X,\F_X(S)$.

We assume that the minimum degree $\delta(G)\geq \theta n$ for some constant $\theta>0$ and that $\psi=1/\log^{2/3}n$ as in \eqref{e0}. Suppose that $\z=n^{-\theta\psi}$.

We partition $V$ as follows: our initial partition $\Pi_0$ consists of $V$ alone. Suppose that we have created a partition $\Pi$, and $X \in \Pi$. We can use the algorithm of  Leighton and Rao, \cite{LR}, to find a cut $(S:\ol S)$ of $X$ such that $\F_X \le \F_X(S) \le c_{LR} \F_X\log n$, where $c_{LR}>0$ constant. If $\F_X(S) \ge \z $,  we do not partition $X$ any further. Otherwise, if $\F_X(S) < \z$, we refine $\Pi$ by  splitting $X$ into $X_1=S$ and $X_2=X\setminus S$. For $\ell=1,2$ let
\beq{part1}{
Y_\ell=Y(X_\ell)=\set{v\in X_\ell:\deg_{X_\ell}(v)\leq \deg_{X_{3-\ell}}(v)}.
}
We replace $X$ in the partition $\Pi$ by the pair 
\beq{part2}{
Z_\ell=(X_\ell\cup Y_{3-\ell})\setminus Y_\ell\text{ for }\ell=1,2.
}
Suppose that, for all $v \in X$, $\deg_X(v)  \ge \b n$, where (see \eqref{degdepth}) $\b=\min_d\b_d$ satisfies $\b > 4 \z^{1/2}$, for $d=O(1)$. If $v \in Y_\ell, \deg_{X\setminus X_\ell}(v) \ge \b n/2$, and thus
\beq{Yival}{
|Y_\ell|\leq \frac{e_X(X_\ell,X\setminus X_\ell)}{\b n/2}\leq\frac{2\z n}{\b}\leq \frac{\z^{1/2} n}{2}.
}
For the second inequality we  used the crude bound, $e_X(X_\ell,X\setminus X_\ell)\leq \z n^2$, which follows from $\F_X(S) < \z$ and \eqref{CondS}.

Continue in this way until the output of the algorithm of \cite{LR} returns a cut $(S, V_i\setminus S)$ such that $\F_{V_i}(S) \ge \z$ for all sets of the partition $\Pi=(V_1,V_2,\ldots,V_{\num})$.
The depth $d_\Pi(V_i)$ of $V_i$ in $\Pi$ is defined as follows: $d_{\Pi_0}(V)=0$ and if $X\in \Pi$ has depth $d$, then its {\em descendants} $Z_1,Z_2$ will both have depth $d+1$.
Suppose that $V_i$ has depth $d$. We claim that $d=O(1)$ and that
\beq{degdepth}{
\min\set{\deg_{V_i}(v):v\in V_i}\geq \b_d n,  \quad \text{where} \quad \b_d=\frac{\theta}{3^d}.
}
If so it follows from \eqref{degdepth}, and $d=O(1)$, that at depth $d$, $|V_i| \ge \b_{d} n=\Om(n)$.

We prove  \eqref{degdepth}  by induction. It is true for $d=0$ and $\b_0=\theta$. If $V_i$ has depth $d+1$ and arises from splitting $X$ at depth $d$ then for $v\in V_i$, provided $d=O(1)$, then from \eqref{Yival}
\beq{degi}{
\deg_{V_i}(v)\geq \frac{\b_d n}{2}-|Y(X_i)|\geq \brac{\frac{\b_d}{2}-\frac{2\z}{\b_d}}n \ge \b_{d+1} n.
}
We also have
\beq{mostdegrees}{
|\set{v\in V_i:\deg_{V_i}(v)\leq \deg(v)-d\z^{1/2}n}|\leq 3 d\z^{1/2}n.
}
This follows by induction. It is true for $d=0$. If $V_i$ has depth $d+1$ and arises from splitting $X$ at depth $d$ then
\multstar{
|\set{v\in V_i:\deg_{V_i}(v)\leq \deg(v)-(d+1)\z^{1/2}n}|\leq\\
 |\set{x\in X:\deg(x)\leq \deg(v)-d\z^{1/2}n}|+2\z^{1/2} n+ \z^{1/2}n.
}
The first term on the RHS is the number of vertices which have low degree at the previous level. The next term counts the at most $2\z^{1/2}n$ vertices which lose at least $\z^{1/2}n/2$ edges, as the cut $(X_1, X_2)$ of $X$ which gave rise to $(V_1,V_2)$ has at most $\z n^2$ edges. The last term comes from \eqref{Yival} and  compensates for the  neighbours of $Y(X_{\ell})$ with at most $\z^{1/2}n/2$ edges in the cut, who lost degree (at most) $|Y(X_{\ell})| \le \z^{1/2}n/2$,  when $Y(X_{\ell})$ was moved out of  $X_\ell$ to obtain $V_i$.

It follows from \eqref{degdepth} and \eqref{mostdegrees} that if $d=O(1)$ and  $X=V_1 \cup V_2$ then,
\beq{sizeS}{
|V_1|, |V_2| \ge \theta n/2.
}
To see this, suppose that $X$ was initially partitioned into $X_1=S,X_2=X\setminus S$ where $\F_X(S)<\z$ and that $V_i=(X_\ell\cup Y_{3-\ell})\setminus Y_\ell$, as in \eqref{part1}, \eqref{part2} above, with $V_i$ replacing $Z_i,i=1,2$.  As $Y_1=X_1\setminus V_1$ and $Y_2=V_1 \setminus X_1$, \eqref{Yival} implies that $|(V_1\setminus X_1)\cup (X_1\setminus V_1)|=o(n)$. This and \eqref{degdepth} implies that $|X_1|=\Omega(n)$. Suppose  that $|V_1|<\theta n/2$. Then \eqref{mostdegrees} implies that there are $|X_1|(1-o(1))$ vertices of $X_1$ of degree at least $\theta n(1-o(1))$ in $X$. So there are at least $\theta|X_1|n(1-o(1))/2$ edges in the cut $X_1:X_2$. And hence,
$$\z>\F_X(X_1)\geq \frac{(\theta|X_1|n(1-o(1))/2)\times d(V)}{|X_1|n\times d(V)}\geq \frac{\theta}3,$$
which is a contradiction.

By \eqref{sizeS} we have that sets at depth $d$ have size at least $\theta n/2$. On the other hand, at least $\theta n/2$ vertices are moved at each partition step, and so sets at depth $d$ have size at  most $n-d\theta n/2$.  This means that $n-d \theta n/2 \ge \theta n/2$, and partitioning must stop when $d< 2/\theta$.

\section{Computing the cover time}
Let $V_1,V_2,\ldots,V_{\num}$ be as in Section \ref{Pgraph}. For each $i$ we add weighted edges to create a multi-graph $H_i$ such that a random walk on $H_i$ corresponds to the visits to $V_i$ of a random walk on $G$. Thus, for each $i$, we define $H_i$ by adding extra edges to $E(V_i)$. If $v,w\in V_i$ then we add an {\em oriented} edge $(v,w)$ and give it a weight $\r_{v,w}$. Here $\r_{v,w}$ is the probability that a walk started at $v$ leaves $V_i$ immediately and returns to $V_i$ at $w$ and we have $\sum_{w\in V_i}\r_{v,w}=\frac{\deg(v)-\deg_{V_i}(v)}{\deg(v)}$. The (unoriented) edges of $G$ contained in $V_i$ will be given weight one. We will use $\w()$ to denote weight in $H_i$.

\begin{Remark}
If we take the random walk $\cW_u=(u=X(0),X(1),...,X(t),...)$ and delete the entries $X(t)$ that are not in $V_i$ then the remaining sequence is a random walk $\cZ_i$ on $H_i$.
\end{Remark}
A random walk $\cZ_i$ on $H_i$ will have steady state $\p_{v,i}=\deg(v)/\deg(V_i)$, $v\in V_i$ and will satisfy the conditions of Theorem \ref{th1}. Indeed, the walk is reversible. Checking detailed balance, we have
\multstar{
\p_{v,i}P_i(v,w)=\frac{\deg(v)}{\deg(V_i)}\brac{\frac{\deg_{V_i}(v)}{\deg(v)}\cdot \frac{
1_{\set{v,w}\in E(H)}}{\deg_{V_i}(v)}+\r_{v,w}}=\\
\frac{1_{\set{v,w}\in  E(H)}+\deg(v)\r_{v,w}}{\deg(V_i)}= \p_{w,i}P_i(w,v),
}
since necessarily, $\deg(v)\r_{v,w}=\deg(w)\r_{w,v}$. This follows from the fact that for any individual walk $W=(x_0=v,x_1,x_2,\ldots,x_k=w)$ from $v$ to $w$ on $G$ and its reversal $\ol W= (x_k=w,x_{k-1},\ldots, x_1,x_0=v)$ we have
$$\p_{v,i}\Pr(W)=\frac{\deg(v)}{\deg(V_i)}\prod_{j=0}^{k-1}\frac{1}{\deg(x_i)} =\frac{\deg(w)}{\deg(V_i)}\prod_{j=1}^{k}\frac{1}{\deg(x_i)}=
\p_{w,i}\Pr(\ol W).$$

To obtain $\deg(v)\r_{v,w}=\deg(w)\r_{w,v}$ we sum over all walks from $v$ to $w$ with interior vertices not in $V_i$.

Now consider the conductance of $H_i$.  In what follows we use the fact that the weight of edges incident with a vertex $v$ in $H_i$ is equal to the degree of $v$ in $G$. Suppose that $S\subseteq V_i$.
According to the definition of $\Phi_{H_i}(S)$ (or rather its extension to graphs with weighted edges),
\[
\Phi_{H_i}(S)=\frac{\w(S,\ol S)\, \w(V_i)}{\deg(S)\, \deg(\ol S)}\ge \frac{e(S, \ol S) \,\deg(V_i)}{ |S||\ol S| n^2} = \F_G(S) \frac{\deg(V_i)}{n^2} \ge \F_G(S) \frac{\theta^2}{3},
\]
assuming $|V_i|$ has at least $\theta n/2$ vertices of degree $\theta n (1-o(1))$.
Thus
\[
\Phi(H_i) \ge \F(G) \frac{\theta^2}{3}\ge \frac{ \F_G(S)}{c_{LR} \log n}\frac{\theta^2}{3} \ge \frac{\z\theta^2 }{3 c_{LR} \log n},
\]
where the last step comes from combining Theorem 2  of Leighton and Rao \cite{LR}, with equation (3) of that paper, that gives a deterministic polynomial algorithm to find a cut $(S:\ol S)$ such that $\F(G) \le \F_G(S) \le c_{LR} \F(G) \log n$.

\subsection{When $T=o(C)$:  Proof of Theorem \ref{th2}}\label{Tsmall}

We see from Theorem \ref{th1} that $\cZ_i$ will have to make a number of steps in some explicit range $(1\pm\e)C_i$ in order to cover $V_i$.
\begin{Remark}
The estimate $C_i$ does not depend significantly on the values $\r_{v,w}$.  We see from Theorem \ref{th1} that up to a factor $(1+o(1))$, the $C_i$ depend only on the degrees of $H_i$. But we can compute close approximations to the $\r_{v,w}$. For this we need to compute the values
\beq{sigma}{
\s_{x,y,i,t}=\Pr(\cW_x(t)=y,\, \cW_x(\t)\notin V_i,1\leq \t\leq t).
}
Given these values, we have
$$\r_{v,w}=\sum_{x,y\notin V_i}P(v,x)P(y,w)\sum_{t\geq 1}\s_{x,y,i,t}.$$
Finally, to compute the values in \eqref{sigma} we simply look at powers of the matrix $Q_i$ that is obtained from $P$ by replacing entries in columns associated with $V_i$ by zeroes.
\end{Remark}

Consider a walk $\cW$ starting in the steady state that walks for $t$ steps. The expected number of visits to $V_i$ is $t\p(V_i)$ and it will be concentrated around this, if the mixing time of $\cW$ is small. For example, Corollary 2.1 of Paulin \cite{Paulin} shows that if $Z_{i,t}$ is the number of visits to $V_i$ then
\beq{Mconc}{
\Pr(|Z_{i,t}-t\p(V_i)|\geq u)\leq \exp\set{-\frac{2u^2}{tT}}.
}
Next let
\beq{defC}{
C=\max\set{\frac{C_i}{\p(V_i)}:i\in [s]}=\Omega(n\log n).
}
Then $C_G=(1\pm\e)C$ if $T=o(C)$. Indeed, putting $t=C,u=\e C$ in \eqref{Mconc} we see immediately that w.h.p. $T_{cov}(u)$ is within a factor $1+o(1)$ of $C$. This immediately gives us a lower bound of $(1-o(1))C$ for the expectation. For the upper bound we use Remark \ref{rem1} in the following way: we know that $\Pr(T_{cov}(u)\in [(1+o(1))C,KC])=o(1)$ and so this range adds $o(C)$ to the expectation. After this, $[KC,\infty]$ adds a negligible amount for large $K$.

This completes our proof of Theorem \ref{th2}.

\subsection{When $T$ is large: Proof of Theorem \ref{th3}}\label{Tlarge}

While a nice formula for the cover time is not necessarily attainable, we claim that we can  deterministically compute quantities that give us a factor $2+o(1)$ estimate for the cover time, in a time polynomial in $n$.

Consider the  $n\times n$ matrix $Q$ where $Q(u,v)=P(u,v)\xi_{\f(u)}$ where the $\xi_i,i\in[s]$ are indeterminate and $\f$ is defined by $u\in V_{\f(u)}$, for $V_i \in \Pi$. Now consider the $t$--th power of $Q$. Then
\multstar{
Q^t(u,v)=\\
\sum_{\t_1+\cdots+\t_s=t}\;\Pr(\cW\text{ goes from $u$ to $v$ in $t$ steps and makes $\t_i$ $V_i$-moves, for $i\in [s]$})\prod_{i=1}^s\xi_i^{\t_i}.
}
Here a $V_i$-move is from a vertex in $V_i$ to any vertex $v \in V$. Note that the number of $V_i$ moves is equal to the number of moves by $\cZ_i$. The cover time of any connected $n$-vertex graph of minimum degree $\delta$ is  $O(n|E|/\delta)$, \cite{KLNS}. When $\delta=\theta n$, $C_G=O(n^2)$. Thus we compute $Q^t$ for $1\leq t\leq n^4$, and observe that this computation can be done in $O(n^7)$ time. Let
$$\k(u,\t,i)\text{ denote the number of steps in $\cW_u$ needed for $\t$ $V_i$-moves}.$$
Next let $C_i^{\pm}=\brac{1\pm\e_2}C_i$ be such that the cover time of  $\cZ_i$ is in $[C_i^-,C_i^+]$ w.h.p., see Lemma \ref{rem2}.

Note that the $C_i$  are given by \eqref{ft1} of Theorem  \ref{th1}, which can be computed in deterministic polynomial time.


Let $U_{i,t}$ denote the set of unvisited vertices of $V_i$ at time $t$. We know from the proof of Theorem \ref{th1} that w.h.p. if $t\leq \k(u,C_i^-,i)$ then $U_{i,t}\neq \emptyset$.
This implies that 
\beq{wt1}{
C_G\geq \max_{u\in V}\E(\max_{i\in [s]}\k(u,C_i^-,i)).
}
For the RHS of \eqref{eth3}, we note that at time $\max_{i\in[s]}\k(u,C_i^-,i)$ the walk $\cW_u$ will be at some vertex $v$ and then after a further $\max_{i\in[s]}\k(v,C_i^-,i)$ steps\footnote{We could write $\max_{i\in[s]}\k(v,C_i^+-C_i^-,i)$ here, but we cannot prove that this is significantly smaller than what we have written.} the walk $\cW_u$ will w.h.p. have spent at least time $2C_i^-$ in $V_i$ for every $i\in[s]$.

Because $2C_i^->C_i^+$, the walk $\cW_u$ will w.h.p. have covered $V$. Thus
\beq{wt2}{
C_G\leq (2+o(1))\max_{u\in V}\E(\max_{i\in[s]}\k(u,C_i^-,i)).
}
and this completes the proof of Theorem \ref{th3}.

\appendix
\section{Proof of Lemma \ref{MainLemma}}
Write
\begin{equation}
\label{uv}
R(z)= R_T(z)+\hR_T(z)+\frac{ \pi_v z^T}{1-z},
\end{equation}
where $R_T(z)$ is given by (\ref{Qs}) and
$$
\hR_T(z)=\sum_{t \ge T} (r_t-\pi_v)z^t
$$
generates the error in using the stationary distribution $\pi_v$ for $r_t$ when $t \ge T$. Similarly,
\begin{equation}
\label{hs}
H(z)= \hH_T(z)+\frac{\pi_vz^T}{1-z}.
\end{equation}
Equation \eqref{4a} implies that the radii of convergence of both $\hR_T$ and $\hH_T$ exceed $1+2\l$. Moreover, for $Z=H,R$ and $|z|\leq 1+\l$, we see from \eqref{4a} that
\begin{equation}\label{12}
|\wh{Z}_T(z)|\leq \frac{\p_v}{2}\sum_{t\geq T}\bfrac{2(1+\l)}{\om}^{\rdown{t/T}}\leq\frac{2(1+\l)T\p_v}{\om}=O(\om^{-2}).
\end{equation}
Using (\ref{uv}), (\ref{hs}) we rewrite $F(z)=H(z)/R(z)$ from  (\ref{gfw}) as  $F(z)=B(z)/A(z)$ where
\begin{eqnarray}
 A(z)&=&\pi_vz^T + (1-z)(R_T(z)+\hR_T(z)),\label{as}\\
B(z)&=&\pi_vz^T +   (1-z)\hH_T(z).\label{bs}
\end{eqnarray}
For real $z \ge 1$ and $Z=H,R$, we have
\[
Z_T(1) \le Z_T(z) \le Z_T(1) z^T.
\]
Let $z=1+\b\pi_v$, where $0\leq \b\leq 1$. Since $T\pi_v\leq \om^{-1}$ we have
\[
Z_T(z)=Z_T(1)(1+\xi_1)\ \text{ where $|\xi_1|\leq (1+\b\p_v)^{T}-1\leq \frac{2\b}{\om}$.}
\]
$T\p_v\leq\om^{-1}$ and $R_v\geq 1$ implies that
$$A(z)=\p_v(1-\b R_v(1+\xi_1))\text{ where $|\xi_1|=O(\om^{-1})$.}$$
It follows that $A(z)$ has a real zero at $z_0$, where
\begin{equation}\label{s0}
z_0= 1+\frac{\pi_v}{R_v(1+\xi_1)}=1+p_v.
\end{equation}
We also see that since $|z_0^T|\leq 1+2\om^{-1}$,
\begin{align*}
A'(z_0)&=T\p_vz_0^{T-1}-(R_T(z_0)+\hR_T(z_0))-p_v(R_T'(z_0)+\hR_T'(z_0))\\
&=O(\om^{-1})-\brac{R_v+O(\om^{-1})+o(\om^{-1})}-o(\p_v)\\
&=-R_v+O(\om^{-1})\\
&\neq 0.
\end{align*}
and thus $z_0$ is a simple zero (see e.g. \cite{BC} p193).
The  value of $B(z)$ at $z_0$ is
\begin{equation}
\label{B0}
 B(z_0)=\p_v\brac{1+O(\om^{-1})+o(\om^{-1})}=\pi_v \brac{1+O(\om^{-1})}\neq 0.
\end{equation}
Thus,
\begin{equation}
\label{BAds}
\frac{B(z_0)}{A'(z_0)}=-\brac{1+\xi_2}p_v\text{ where }|\xi_2|=O(\om^{-1}).
\end{equation}
Thus  (see e.g. \cite{BC} p195) the principal part of the Laurent expansion of $F(z)$ at $z_0$  is
\begin{equation}\label{ps}
f(z)= \frac{B(z_0)/A'(z_0)}{z-z_0}.
\end{equation}
To approximate the coefficients of the generating function $F(z)$, we now use  a standard technique  for the asymptotic expansion  of power series (see e.g.\cite{Wi} Theorem 5.2.1).

We prove below that $F(z)=f(z)+g(z)$, where $g(z)$ is analytic in $C_{\l}=\{|z|\leq 1+\l\}$ and that
$$M=\max_{z \in C_{\l}} |g(z)|=O(\om^{-1}).$$

Let $a_t=[z^t]g(z)$, then (see e.g.\cite{BC} p143), $a_t=g^{(t)}(0)/t!$. By the Cauchy Inequality (see e.g. \cite{BC} p130) we see that $|g^{(t)}(0)| \le Mt!/(1+\l)^t$ and thus
\[
|a_t| \le \frac{M}{(1+\l)^t}\leq Me^{-t\l/2}.
\]
As $[z^t]F(z)=[z^t]f(z)+[z^t]g(z)$  and $[z^t] 1/(z-z_0)=-1/z_0^{t+1}$ we have
\begin{equation}\label{ans}
[z^t]F(z)= \frac{-B(z_0)/A'(z_0)}{z_0^{t+1}}+\eta_1(t)\text{ where }|\eta_1(t)|\leq Me^{-t\l/2}.
\end{equation}
Thus,  we obtain
\[
[z^t]F(z)
 =
\frac{(1+\xi_2)p_v}{(1+p_v)^{t+1}} + \eta_1(t).
\]
Now
$$\Pr(\ul A_t(v))= \sum_{\t > t} f_{\t}(u \rat v) = \sum_{\t > t} \brac{\frac{(1+\xi_2)p_v}{(1+p_v)^{\t+1}} + \eta_1(\t)} =\frac{1+\xi_2}{(1+p_v)^{t+1}}+\eta_2(t),$$
where
$$\eta_2(t)=\sum_{\t>t}\eta_1(t)\leq \frac{Me^{-\l t/2}}{1-e^{-\l/2}}=o(Te^{-\l/2}).$$
This completes the proof of (\ref{frat}).

Now $M=\max_{z \in C_{\l}} |g(z)|\leq \max |f(z)|+\max |F(z)|=O(T\p_v)+\max |F(z)|=O(\om^{-1})+\max |F(z)|$, where $F(z)=B(z)/A(z)$. On $C_{\l}$ we have, using \eqref{12}-\eqref{bs},
\[
|F(z)| \leq \frac{\p_vz^T+o(\p_v)}{\p_vz^T+\l(|R_T(z)|-O(\om^{-2}))}=O\bfrac{\p_vz^T}{T^{-1}R_v}=O(\om^{-1}).
\]

We now prove that $z_0$ is the  only   zero of $A(z)$ inside the circle $C_{\l}$ and this implies that $F(z)-f(z)$ is
analytic inside $C_\l$. We  use  Rouch\'e's Theorem (see e.g. \cite{BC}), the statement of which is as follows:
{\em Let two functions $\f(z)$ and $\g(z)$ be analytic inside and on a simple closed contour $C$. Suppose that $|\f(z)|>|\g(z)|$ at each point of $C$, then $\f(z)$ and $\f(z)+\g(z)$ have the same number of zeroes, counting
multiplicities, inside $C$.}

Let the functions $\f(z),\g(z)$ be given by  $\f(z)=(1-z)R_T(z)$ and $\g(z)= \p_vz^T+(1-z)\hR_T(z)$.
$$\frac{|\g(z)|}{|\f(z)|}\leq\frac{\p_v(1+\l)^T}{\l\theta}+\frac{|\hR_T(z)|}{\theta}=o(1).$$
As $\f(z)+\g(z)=A(z)$ we conclude that $A(z)$ has only one zero inside the circle $C_{\l}$. This is the simple
zero at $z_0$.

\end{document}